\documentclass[11pt]{amsart}

\usepackage{geometry}
\usepackage[all]{xy}
\usepackage{epstopdf}
\usepackage{amssymb,amsmath,amsfonts,amsthm,graphicx,enumerate,amscd}
\usepackage{mathrsfs}

\theoremstyle{plain}
\newtheorem{theorem}{Theorem}[section]

\newtheorem{corollary}[theorem]{Corollary}
\newtheorem{lemma}[theorem]{Lemma}
\newtheorem{proposition}[theorem]{Proposition}

\theoremstyle{definition}

\newtheorem{remark}[theorem]{Remark}

\newtheorem{definition}[theorem]{Definition}

\title{Picard group of isotropic realizations of twisted Poisson manifolds}
\author{Chi-Kwong Fok}
\date{September 1, 2015}

\begin{document}
\begin{abstract}
	Let $B$ be a twisted Poisson manifold with a fixed tropical affine structure given by a period bundle $P$. In this paper, we study the classification of almost symplectically complete isotropic realizations (ASCIRs) over $B$ in the spirit of \cite{DD}. We construct a product among ASCIRs in analogy with tensor product of line bundles, thereby introducing the notion of the Picard group of $B$. We give descriptions of the Picard group in terms of exact sequences involving certain sheaf cohomology groups, and find that the `N\'eron-Severi group' is isomorphic to $H^2(B, \underline{P})$. An example of an ASCIR over a certain open subset of a compact Lie group is discussed.
\end{abstract}
\maketitle
\tableofcontents
\section{Introduction}
A completely integrable Hamiltonian system is a symplectic $2n$-dimensional manifold with $n$ independent smooth functions being pairwise in involution, i.e. the Poisson brackets (induced by the symplectic structure) of any pair of them vanish. On the one hand, the Liouville-Arnold theorem, a well-known basic result on the local structure of completely integrable Hamiltonian systems, asserts the existence of local action-angle coordinates. Duistermaat, on the other hand, first explored the global structure in \cite{Dui} in the early 80s, where he worked in the slightly more general setting of Lagrangian fiber bundles. He introduced two topological invariants, namely monodromy and Chern classes. The triviality of the former quantity gives a necessary (but not sufficient) condition for the existence of global action coordinates, while the vanishing of the latter, together with a certain condition on the symplectic form, amount to the existence of global angle coordinates. Duistermaat's foundational work has since elicited generalizations in mainly two contexts, superintegrable systems (mechanical systems where the number of independent smooth functions in involution is less than half of the dimensions of the systems) and almost symplectic systems, to be explained below.

Dazord and Delzant in \cite{DD} studied, in the spirit of \cite{Dui}, the global structure of symplectically complete isotropic realizations (SCIR) which correspond to superintegrable systems. In particular, they defined the Lagrangian class, an invariant finer than the Chern class and the vanishing of which is equivalent to the existence of global angle coordinates. Moreover, they introduced what is now called Dazord-Delzant homomorphism of certain sheaf cohomology groups and used it to give a sufficient and necessary topological condition for a torus bundle over a Poisson manifold to have a compatible symplectic form so as to be an SCIR. More details on their results can also be found in \cite{Va}.

In recent years there has been a growing interest in the study of nonholonomic systems, and as a preliminary step of investigation in this context almost symplectic systems are considered. For instance, in \cite{FS} Fass\`o and Sansonetto studied almost symplectic integrable Hamiltonian systems (IASHS, see Definition \ref{IASHS}), which are basically SCIRs except that they are almost symplectic. They obtained a generalization of Liouville-Arnold Theorem under the condition of the existence of strongly Hamiltonian vector fields (See Theorem \ref{IASHSmain}). The global theory in this direction in the spirit of Duistermaat/Dazord-Delzant was first outlined by Sjamaar in \cite{Sj}, where the special case of almost Lagrangian fiber bundles were considered. Sansonetto and Sepe, in the recent paper \cite{SS}, generalized the Dazord-Delzant theory in both contexts of superintegrable systems and almost symplectic systems. They studied what we call in this paper almost symplectically complete isotropic realizations (ASCIR, see Definition \ref{ASCIR}), which are equivalent to IASHSs. They describe the tropical affine structure (called transversally integral affine structure in \cite{SS}. See Definition \ref{tropical}) coming from ASCIRs and Chern class, and use the sheaf cohomology groups (which are exactly the same as those in \cite{DD}) to describe Dazord-Delzant homomorphism in the case of ASCIRs. They brought twisted Poisson structures into play, seeing as the base space of an ASCIR inherits a twisted Poisson structure (cf. Theorem \ref{IASHSmain}). Motivated by string theory, twisted Poisson manifolds were first introduced in \cite{SW}, and provide the framework for the study of such nonholonomic systems as the Veselova systems and the Chaplygin sphere (see \cite{BG-N}). 

In this paper, following the ideas outlined in \cite{Sj}, we continue the study of ASCIRs. In particular, we show, by a symplectic reduction argument, that there is a `tensor product' of ASCIRs over a fixed twisted Poisson manifold $B$ with a given tropical affine structure, analogous to the tensor product of line bundles. The resulting group of isomorphism classes of ASCIRs, which we call the Picard group of $B$, admits descriptions in terms of exact sequences of certain sheaf cohomology groups of $B$ as in Theorem \ref{classification}. We also consider the twisted Poisson structure on an open subset of a compact Lie group described in \cite{SW}, and exhibit an example of an ASCIR over it.

%From the short exact sequence in part (\ref{NS}) of Theorem \ref{classification}, we can think of the image of $\Omega_\mathcal{F}^2/d\Omega^1_\mathcal{F}$ in the Picard group as the subgroup of degree zero invertible sheaves on schemes in algebraic geometry, whereas $H^2(B, \mathcal{P})$ plays the role of N\'eron-Severi group. 

The organization of this paper is as follows. Section \ref{intalsymham} reviews the analogue of Liouville-Arnold Theorem for IASHSs as in \cite{FS} and the notion of twisted Poisson manifolds. ASCIRs are defined and shown to be equivalent to IASHSs. In Sections \ref{tortropaff} and \ref{chern} we define certain topological structures ASCIRs possess, namely torsors and tropical affine structures. We also review Chern classes and introduce certain sheaf cohomology groups to set the scene for classifying ASCIRs over a fixed twisted Poisson manifold with a given tropical affine structure. In Section \ref{picard}, we explain in detail the construction of the tensor product of ASCIRs and define the Picard group. We restate the main result of \cite{SS}, which gives, by means of the Dazord-Delzant homomorphism, the necessary and sufficient condition for a torsor to be an ASCIR with a certain twisting 3-form. Our main results on the classification of ASCIRs and the description of the Picard group are also given and proved. In Section \ref{twistedlieex} we discuss the example of an ASCIR over an open subset of a compact Lie group with the twisted Poisson structure as in \cite{SW}. Section \ref{remarks} points out how the present work fits into the more general framework of integration of twisted Poisson and Dirac structures, and Picard groups in Poisson geometry.

\textbf{Acknowledgments}. The author is indebted to his advisor, Reyer Sjamaar, for his patient guidance and generously sharing his ideas which constitute a major part of this paper. He would also like to thank the referees for their detailed and critical comments and suggestions for improvements.

\section{Integrable almost symplectic Hamiltonian systems}\label{intalsymham}
In \cite{FS}, the notion of strong Hamiltonianicity is introduced and exploited to generalize the Liouville-Arnold theorem on local action-angle coordinates in the context of integrable almost symplectic Hamiltonian systems (cf. Definition \ref{IASHS}), which are mechanical systems that, in terms of generality, lie between Hamiltonian ones and nonholonomic ones (cf. \cite{SS}). In what follows we shall first recall the relevant notions in this general setting before stating their results.

\begin{definition}\label{IASHS}
		An \emph{almost symplectic manifold} is a pair $(M, \omega)$, where $M$ is a smooth manifold and $\omega\in\Omega^2(M)$ is a non-degenerate 2-form. Any $\omega$ as above is referred to as an \emph{almost symplectic form}. 
\end{definition}
\begin{definition}(cf. \cite{FS})
	A function $f\in C^\infty(M)$ is strongly Hamiltonian if the associated vector field $X_f:=(\omega^\sharp)^{-1}(df)$ satisfies $\mathcal{L}_{X_f}\omega=0$.
\end{definition}
\begin{definition}(cf. \cite{FS} and \cite[Definition 6]{SS})
		We call $\pi: M^{2d}\to B^k$ an \emph{integrable almost symplectic Hamiltonian system} if the following conditions hold
		\begin{enumerate}
			\item $\pi$ is a surjective submersion with compact and connected fibers.
			\item $M$ has an almost symplectic form $\omega$.
			\item Fibers of $\pi$ are isotropic.
			\item For any point $b\in B$, there exists a $\pi$-saturated neighborhood $U$ such that there exist linearly independent, strongly Hamiltonian vector fields $Y_1, \cdots, Y_n$ (here $n:=2d-k$) in $U$ which are tangent to the fibers of $\pi$.
		\end{enumerate}
	
\end{definition}
\begin{definition}[cf. \cite{SW}]\label{twistedPoisson}
	A \emph{twisted Poisson manifold} is a pair $(B, \Pi)$ where $\Pi$ is a bivector field such that $[\Pi, \Pi]\in (\bigwedge\nolimits^3\Pi^\sharp)(Z^3(B))$, where $\Pi^\sharp: T^*B\to TB$ is defined by $\Pi^\sharp(\alpha)(\beta)=\Pi(\alpha, \beta)$. An $\eta$-\emph{twisted Poisson manifold} is a triple $(B, \Pi, \eta)$ if $[\Pi, \Pi]=\bigwedge\nolimits^3\Pi^\sharp(\eta)$. If $\Pi$ is of constant rank then we say the twisted Poisson structure is \emph{regular}. We also define $X_f:=-\Pi^\sharp(df)$. 
\end{definition}
\begin{remark}\label{twistedpoissonrmk}
	The term $\bigwedge\nolimits^3\Pi^\sharp(\eta)$ measures how $\Pi$ fails to satisfy the Jacobi identity. The condition $[\Pi, \Pi]=\bigwedge^3\Pi^\sharp(\eta)$ amounts to requiring that for any smooth functions $f$, $g$ and $h$, 
	\begin{eqnarray}\label{twistedpoissoneq}\{\{f, g\}, h\}+\{\{g, h\}, f\}+\{\{h, f\}, g\}+\eta(X_f, X_g, X_h)=0.\end{eqnarray}
	Similar to the Poisson case, a twisted Poisson structure induces a (in general singular) foliation in which each leaf is equipped with an almost symplectic form, whose exterior differential is the restriction of the twisting form $\eta$ to the leaf (cf. \cite{SW}). In particular, if two twisting forms $\eta_1$ and $\eta_2$ correspond to the same twisted Poisson bivector field $\Pi$, then $\eta_1-\eta_2$ vanishes on restriction to each almost symplectic leaf. Conversely, if $\eta_1$ is a twisting 3-form corresponding to $\Pi$, and $\eta_2$ another 3-form such that $\eta_1-\eta_2$ vanishes on restriction to each leaf, then $\eta_2$ is also a twisting 3-form corresponding to $\Pi$. That is because the vector field $X_f$ is tangent to the leaves, and as a result $(\eta_1-\eta_2)(X_f, X_g, X_h)=0$ and $\eta_2$ satisfies Equation (\ref{twistedpoissoneq}) as well. 
\end{remark}
Throughout this paper, $(B, \Pi)$ is always a regular twisted Poisson manifold with almost symplectic foliation denoted by $\mathcal{F}$, unless other specified.
\begin{definition}\label{tropical}
	Let $(B, \mathcal{F})$ be a manifold endowed with a regular foliation. A \emph{tropical affine structure} transversal to the leaves of $B$ is a system of foliated atlas $\mathcal{A}=\{(U_\alpha, p_\alpha)\}$ for $\mathcal{F}$ consisting of submersions $p_\alpha$ which locally define the foliation of the leaf space such that on $U_{\alpha\beta}$, the change of coordinates is given by 
	\[p_\beta=A_{\alpha\beta}p_\alpha+z_{\alpha\beta}\]
	for $A_{\alpha\beta}\in GL(n, \mathbb{Z})$ and a constant vector $z_{\alpha\beta}\in\mathbb{R}^n$. In other words, the transition functions take values in the \emph{tropical affine group} $\textbf{Trop}(n):=GL(n, \mathbb{Z})\ltimes \mathbb{R}^n$.
\end{definition}
\begin{theorem} (cf. \cite[Section 3]{SS} and \cite[Proposition 4, 5 and 6]{FS}) \label{IASHSmain}
 $\pi: M^{2d}\to B^k$ is an integrable almost symlectic Hamiltonian system, then we have
	\begin{enumerate}
		\item $\pi$ is a $T^n$-fiber bundle, where $T^n$ is the $n$-dimensional torus.
		\item Local action coordinates $a_1, \cdots, a_n$, local angle coordinates $\alpha_1, \cdots, \alpha_n$ and $b_1, \cdots, b_{2d-2n}$ exist. $\omega$ in these coordinates has the following representation
		\[\sum_{i=1}^nda_i\wedge d\alpha_i+\frac{1}{2}\sum_{i<j}A_{ij}da_i\wedge da_j+\sum_{i, j}B_{ij}da_i\wedge db_j+\frac{1}{2}\sum_{i<j}C_{ij}db_i\wedge db_j\]
		where $A_{ij}$, $B_{ij}$ and $C_{ij}$ are independent of the angle coordinates, and $C_{ij}$ is nonsingular. In particular, $d\omega=\pi^*\eta$ for some closed 3-form $\eta$ of $B^k$.
		\item $B$ is an $\eta$-twisted regular Poisson manifold of rank $2d-2n$, where the twisted Poisson structure $\Pi$ is coinduced by $\omega$, i.e. the twisted Poisson bracket $\{\cdot, \cdot\}_B$ satisfies
		\[\pi^*\{f, g\}_B=\omega(X_{\pi^*f}, X_{\pi^*g})\]
		\item The almost symplectic leaves are locally the level sets of the action coordinates $a_1, \cdots, a_n$ on $B$, which are locally the Casimirs of the twisted Poisson structure of $B$.
		\item The local action coordinates give rise to a tropical affine structure transversal to the almost symplectic leaves of $B$. 
	\end{enumerate}
\end{theorem}

In Theorem \ref{IASHSmain}, we start with the almost symplectic and strongly Hamiltonian structures of the integrable system $M$, and deduce the various compatible structures of $B$, i.e. the regular twisted Poisson structure and the tropical affine structure. This is the geometric mechanical approach taken in \cite{FS}. In the rest of this paper, we shall take the theoretical approach as in \cite{DD} in analogy to the integration of \emph{Lie algebroids} (for more explanation see Section \ref{remarks}) by starting with the regular twisted Poisson and tropical affine structures on $B$ and investigate what integrable almost symplectic Hamiltonian systems $B$ can support (or what torus bundles over $B$ support compatible integrable almost symplectic Hamiltonian structures). This question was first answered in \cite{DD} in the genuine symplectic case and later by \cite{SS} in the almost symplectic case by generalizing the arguments in \cite{DD}. 

\begin{definition}\label{ASCIR}
	An \emph{almost symplectically complete isotropic realization} (ASCIR) $(M, \omega)$ over a regular twisted Poisson manifold $(B, \Pi)$ is a surjective submersion $\pi: M\to B$ whose fibers are compact and connected and such that $\pi$ is an almost Poisson map (here the almost Poisson structure on $M$ is induced by $\omega$), and $d\omega=\pi^*\eta$ for some closed 3-form $\eta$, which we call the \emph{twisting form} of the ASCIR. 
\end{definition}
We note that the almost Poisson bi-vector induced by the almost symplectic structure of $M$ satisfies $[\Pi_M, \Pi_M]=\bigwedge^3\Pi^\sharp_M(d\omega)$. $\pi$ being an almost Poisson map, the Poisson bi-vector $\Pi$ of $B$ satisfies $[\Pi, \Pi]=\bigwedge^3\Pi^\sharp(\eta)$. So $B$ is $\eta$-twisted if $d\omega=\pi^*\eta$. The following Proposition justifies that there is nothing lost in taking the aforementioned `theoretical' approach.
\begin{proposition}\label{IASHSASCIR}
	$(M, \omega)$ is an ASCIR over $(B, \Pi, \eta)$ if and only if $\pi: M\to B$ is an integrable almost symplectic Hamiltonian system.
\end{proposition}
\begin{proof}
	One direction is simply part of the content of Theorem \ref{IASHSmain}. Now suppose that $(M, \omega)$ is an ASCIR over $(B, \Pi, \eta)$. Conditions (1) and (2) in Definition \ref{IASHS} are also part of Definition \ref{ASCIR}. Note that $\{\pi^*f, \pi^*g\}=d\pi^*g(X_{\pi^*f})=\pi^*(dg(\pi_*X_{\pi^*f}))$, while $\pi^*\{f, g\}=\pi^*(dg(X_f))$. Since $\pi$ is an almost Poisson and submersive map, $dg(\pi_*X_{\pi^*f})=dg(X_f)$ for any $f, g\in C^\infty(B)$. So $\pi_*X_{\pi^*f}=X_f$. Let $(a_1, \cdots, a_n)$ be local Casimir coordinate functions on an open neighborhood $U$ of $B$ where the level sets are the almost symplectic leaves. Then $\pi_*X_{\pi^*a_i}=X_{a_i}=0$. It follows that $\{X_{\pi^*a_i}\}_{i=1}^n$ spans the tangent bundle of the fibers over $U$. Moreover, $X_{\pi^*a_i}$ is strongly Hamiltonian because
	\begin{align*}
		\mathcal{L}_{X_{\pi^*a_i}}\omega&=\iota_{X_{\pi^*a_i}}d\omega+d\iota_{X_{\pi^*a_i}}\omega\\
		                                                       &=\iota_{X_{\pi^*a_i}}\pi^*\eta+d(d(\pi^*a_i))\\
		                                                       &=0
	\end{align*}
So $\pi: M\to B$ satisfies condition (4) in Definition \ref{IASHS}. Condition (3) follows from the fact that 
\[\omega(X_{\pi^*a_i}, X_{\pi^*a_j})=\{\pi^*a_i, \pi^*a_j\}_M=\pi^*\{a_i, a_j\}_B=0\]
\end{proof}

\begin{remark}
	Varying the twisting 3-form $\eta$ of the twisted Poisson manifold $(B, \Pi, \eta)$ by $d\gamma$, where $\gamma$ is a 2-form vanishing on each almost symplectic leaf, not only leaves the twisted Poisson structure $\Pi$ and the tropical affine structure unchanged, but also does not change the topological structure of any ASCIR over $(B, \Pi, \eta)$, for if $(M, \omega)$ is an ASCIR over $(B, \Pi, \eta)$, then $(M, \omega+\pi^*\gamma)$ is an ASCIR over $(B, \Pi, \eta+d\gamma)$. 
\end{remark}

\section{Lie group bundles, torsors, tropical affine structures and period bundles}\label{tortropaff}
\begin{definition}
	\begin{enumerate}
		\item A bundle of Lie groups $\mathcal{G}$ over $B$ is a locally trivial fiber bundle with fibers being a Lie group $G$ such that the transition functions take values in the automorphism group of $G$.
		\item An action of the bundle of Lie groups $\mathcal{G}\to B$ on a fiber bundle $\pi: M\to B$ is a smooth map $\mathcal{A}: \mathcal{G}\times_B M\to M$ such that 
		\begin{enumerate}
			\item $\pi(\mathcal{A}(g, m))=\pi(m)$,
			\item $\mathcal{A}(1, m)=m$, and
			\item $\mathcal{A}(g, \mathcal{A}(h, m))=\mathcal{A}(gh, m)$.
		\end{enumerate}
		$M$ is a $\mathcal{G}$-\emph{torsor} if each fiber of $\mathcal{G}$ acts on the corresponding fiber of $M$ freely and transitively, i.e. if $(\mathcal{A}, \pi_2): \mathcal{G}\times_B M\to M\times_B M$ is a diffeomorphism.
	\end{enumerate}
\end{definition}
An example of bundles of Lie groups is any Euclidean vector bundle, whose fibers are isomorphic to $\mathbb{R}^n$. A principal $G$-bundle over $B$ furnishes an example of a torsor, where the bundle of Lie groups acting on it is the trivial bundle $B\times G$. On the other hand, if $M$ is a $\mathcal{G}$-torsor, then it is locally a principal $G$-bundle. Note that a bundle of Lie groups is an example of a \emph{Lie groupoid}, where the source and target maps are identical.

Suppose $(M, \omega)$ is an ASCIR over $(B, \Pi, \eta)$. The Euclidean vector bundle
\[\nu^*\mathcal{F}:=\{(b, \alpha)\in T^*_bB|\alpha|_{T_bL_b}=0, b\in B\}\]
where $L_b$ is the leaf passing through $b$ is a bundle of abelian Lie groups (in fact, it is a bundle of abelian Lie algebras and a Lie algebroid. See \cite[P2]{SS}). It gives a natural action of bundle of Lie groups in the following way. A form $\alpha\in \nu^*_b\mathcal{F}$ gives rise to a vector field $\omega^\sharp(\pi^*\alpha)\in \Gamma(TM|_{\pi^{-1}(b)})$. Since $\alpha=d_bf$ for some function $f$ defined in a neighborhood of $b$ and constant on the almost symplectic leaves of $B$, $\omega^\sharp(\pi^*\alpha)=X_{\pi^*f}$. On the other hand, $\pi_*(X_{\pi^*f})=X_f$ (see the proof of Proposition \ref{IASHSASCIR}),  and $f$ is a local Casimir of the twisted Poisson structure, so $X_f=0$ and $X_{\pi^*f}$ is tangent to the fibers. In particular, $\omega^\sharp(\pi^*\alpha)$ is a vector field of $\pi^{-1}(b)$, which is compact by assumption. Therefore $\omega^\sharp(\pi^*\alpha)$ integrates to a time 1 flow, which we denote by $\varphi_b(\alpha): \pi^{-1}(b)\to \pi^{-1}(b)$. Letting $b$ vary smoothly yields a smooth action of $\nu^*\mathcal{F}$ on $M$. Let $P$ be the stabilizer subbundle of this action. This is a smooth subbundle, and the quotient $\nu^*\mathcal{F}/P$ is a bundle of abelian Lie groups (cf. \cite[Theorem 2 (1) and (3)]{SS}, which is proved using the same arguments in the symplectic case in \cite[Theorem 8.15]{Va}).
\begin{definition}
	We call $P$ the \emph{period bundle} of the ASCIR $(M, \omega)$. 
\end{definition}
We can say something more about the action $\varphi$ and the period bundle $P$. Note that the map $\nu_b^*\mathcal{F}\to T_m\pi^{-1}(b)$ defined by $\alpha\mapsto \omega^\sharp(\pi^*\alpha)$ is an isomorphism. This, together with the assumption that $\pi^{-1}(b)$ is connected, implies that the action $\varphi_b$ is transitive and locally free. It follows that the stabilizer $P_b$ is a discrete subgroup of $\nu_b^*\mathcal{F}$ and $\pi^{-1}(b)\cong\nu_b^*\mathcal{F}/P_b$. We have that $P_b\cong\mathbb{Z}^n$ as $\pi^{-1}(b)$ is compact. To summarize, 

\begin{proposition}(\cite[Theorem 2 (4)]{SS})\label{ASCIRtorsor}
	If $(M, \omega)$ is an ASCIR over a regular twisted Poisson manifold $(B, \Pi, \eta)$, then $M$ is a $\nu^*\mathcal{F}/P$-torsor, where $P$, the period bundle of $M$, is a $\mathbb{Z}^n$-bundle. 
\end{proposition}

\begin{proposition}(\cite[Proposition 8.16 (i)]{Va} and \cite[Proposition 2]{SS})
	Let $\alpha\in\Gamma(\nu^*\mathcal{F})$, and $\varphi$ the action of $\nu^*\mathcal{F}$ on an ASCIR $M$ as in the above discussion. Then 
	\[\varphi(\alpha)^*\omega=\omega+\pi^*d\alpha\]
\end{proposition}
%\begin{proof}
%	\begin{align*}
%		&\frac{d}{dt}(\text{exp}(t\omega^\sharp(\pi^*\alpha)))^*\omega\\
%		=&\text{exp}(t\omega^\sharp(\pi^*\alpha))^*\mathcal{L}_{\omega^\sharp(\pi^*\alpha)}\omega\\
%		=&\text{exp}(t\omega^\sharp(\pi^*\alpha))^*(d\iota_{\omega^\sharp(\pi^*\alpha)}\omega+\iota_{\omega^\sharp(\pi^*\alpha)}d\omega)\\
%		=&\text{exp}(t\omega^\sharp(\pi^*\alpha))^*(\pi^*d\alpha+\iota_{\omega^\sharp(\pi^*\alpha)}\pi^*\eta)\\
%		=&(\pi\circ \text{exp}(t\omega^\sharp(\pi^*\alpha)))^*d\alpha\\
%		=&\pi^*d\alpha
%	\end{align*}
%Integrating both sides from 0 to 1, we get the desired equation. 
%\end{proof}
\begin{corollary}(\cite[Proposition 8.16 (ii)]{Va} and \cite[Corollary 1]{SS})\label{closedP}
	Sections of $P$ are closed 1-forms which vanish on restriction to the almost symplectic leaves of $B$. 
\end{corollary}

\begin{definition}
		Let $\omega_{\nu^*\mathcal{F}}$ be the canonical 2-form of $\nu^*\mathcal{F}$ which satisfies $\xi^*\omega_{\nu^*\mathcal{F}}=d\xi$ for $\xi\in\Gamma(\nu^*\mathcal{F})$.  
\end{definition}
Note that $\omega_{\nu^*\mathcal{F}}$ is closed, but not nondegenerate unless the dimension of the leaves is 0, i.e. $\nu^*\mathcal{F}=T^*B$. Moreover, $\omega_{\nu^*\mathcal{F}}$ vanishes on restriction to any local section of $P$ by Corollary \ref{closedP} and $\xi^*\omega_{\nu^*\mathcal{F}}=d\xi$. Hence $\omega_{\nu^*\mathcal{F}}$ descends to a 2-form on $\nu^*\mathcal{F}/P$.

\begin{definition}
	Let $T$ be $\nu^*\mathcal{F}/P$, and $\omega_T$ the 2-form obtained by descending $\omega_{\nu^*\mathcal{F}}$ to $T$. 
\end{definition}

By Theorems \ref{IASHSmain} and \ref{IASHSASCIR}, an ASCIR admits local action coordinates $(a_1, \cdots, a_n)$ which give rise to a tropical affine structure of $B$ transversal to $\mathcal{F}$. One can glue the locally trivial $\mathbb{Z}^n$-bundles $U\times(\mathbb{Z}da_1\oplus\cdots\oplus\mathbb{Z}da_n)$ with transition functions being the linear part of the change of the action coordinate functions to get the period bundle $P$. Conversely, the tropical affine structure of $B$ can be recovered from the period bundle $P$ of any ASCIR: we may take a local frame $(s_1, \cdots, s_n)$ of $P$ over a small enough neighborhood $U$ of $B$ so that by Poincar\'e Lemma and Corollary \ref{closedP}, $s_i$ is locally exact. Any choice of the primitives of the local frame serves as local action coordinates. The linear part of the transition functions of these primitives is exactly the transition functions of $P$, while the translational part is accounted for by the ambiguities of different choices of the primitives, which differ by a constant vector. So the primitives give a tropical affine structure of $B$. To summarize,

\begin{proposition}(\cite[Proposition 4]{SS})\label{tropaffstr}
	$P$ and the tropical affine structure give rise to each other.
\end{proposition}

\section{Chern classes and some sheaf cohomology groups}\label{chern}
In this section, we present a topological invariant, the \emph{Chern class}, and certain sheaf cohomology groups to classify a subcategory of ASCIRs over a fixed regular twisted Poisson manifold $B$ with a given tropical affine structure following \cite{Dui} and \cite{DD}. %The vanishing of the class in a certain sheaf cohomology defined for a subcategory of ASCIRs (see Definition ) has the classical mechanical interpretation of measuring the obstruction of the existence of global angle coordinates. 
\begin{definition}\label{sheaves}
	Let $(B, \Pi)$ be a regular twisted Poisson manifold with a tropical affine structure transversal to $\mathcal{F}$ induced by the period bundle $P$. We define, on $B$, 
	\begin{enumerate}
		\item $\Omega^k_\mathcal{F}:=$the sheaf of differential $k$-forms which vanish on each almost symplectic leaf, if $k>1$. $\Omega^0_\mathcal{F}:=$the sheaf of smooth functions which are constant on each almost symplectic leaf. 
		\item $\mathcal{Z}^k_\mathcal{F}:=$the sheaf of closed $k$-forms in $\Omega^k_\mathcal{F}$.
		\item $\mathcal{T}:=$the sheaf of smooth sections of $T$.
		\item $\mathcal{P}:=$the sheaf of locally constant sections of $P$. 
		\item $\mathcal{O}_P:=$the structure sheaf of $B$ with tropical affine structure $P$, defined by 
		\[\mathcal{O}_P(U):=\{f\in C^\infty(U)|df\in\mathcal{P}(U)\}\]
		\item $\mathcal{K}_\mathcal{F}:=\mathcal{Z}^1_\mathcal{F}/\mathcal{P}$.
	\end{enumerate}
\end{definition}
It is easy to see that the various sheaves in Definition \ref{sheaves} are linked in the following web of short exact sequences, which has been known since the work \cite{DD}. 
\begin{displaymath}
	\xymatrix{&0\ar[d]&&&\\&\mathbb{R}\ar[d]&0\ar[d]&0\ar[d]&\\&\mathcal{O}_P\ar[d]^d&\mathcal{Z}_\mathcal{F}^1\ar[r]\ar[d]&\mathcal{K}_\mathcal{F}\ar[d]\ar[r]&0\\0\ar[r]&\mathcal{P}\ar[ur]\ar[r]\ar[d]&\Omega_\mathcal{F}^1\ar[r]\ar[d]^d&\mathcal{T}\ar[r]\ar[d]^{d_{P}}&0\\&0&\mathcal{Z}_\mathcal{F}^2\ar[d]&\mathcal{Z}_\mathcal{F}^2\ar[d]&\\&&0&0&}
\end{displaymath}

As shown in Proposition \ref{ASCIRtorsor}, any ASCIR is a $T$-torsor. Therefore, as a first step to addressing the classification problem, we may forget the almost symplectic structure for the time being and classify $T$-torsors over $B$. Let $M$ be a $T$-torsor over $B$, and $\mathcal{U}=\{U_\alpha\}$ be an open cover of $B$ such that $U_\alpha$ admits a smooth local section $s_\alpha$. Because of the transitivity and the freeness of the $T$-action on $M$, we can define $t_{\alpha\beta}\in \mathcal{T}(U_{\alpha\beta})$ by 
\[\varphi(t_{\alpha\beta})s_\beta=s_\alpha\]
$\{t_{\alpha\beta}\}$ defines a Cech 1-cocyle and a class $[t]$ in the sheaf cohomology $H^1(B, \mathcal{T})$, which gives the obstruction for $M$ to possess a global section. The short exact sequence
\[0\longrightarrow \mathcal{P}\longrightarrow\Omega_\mathcal{F}^1\longrightarrow\mathcal{T}\longrightarrow 0\]
induces the long exact sequence of sheaf cohomology groups
\[\cdots\longrightarrow H^1(B, \Omega_\mathcal{F}^1)\longrightarrow H^1(B, \mathcal{T})\stackrel{\delta}{\longrightarrow}H^2(B, \mathcal{P})\longrightarrow H^2(B, \Omega_\mathcal{F}^1)\longrightarrow\cdots\]
Since $\Omega_\mathcal{F}^1$ is a fine sheaf, $\delta$ is an isomorphism.

\begin{definition}[cf. \cite{Dui}]
	The class $c(M):=\delta([t])\in H^2(B, \mathcal{P})$ is called the Chern class of $M$ as a $T$-torsor.
\end{definition}
A $T$-torsor has vanishing Chern class if and only if it possesses a global section or, equivalently, is isomorphic to $T$ as a $T$-torsor. The Chern classes measure the obstruction of the existence of a global section of a $T$-torsor, so they are purely topological in nature and independent of the almost symplectic structure, if any, of the $T$-torsor. They provide the complete invariant of the classification of $T$-torsors on $B$. In category theoretic language, if $\text{Tor}_{(B, T)}$ denotes the category of $T$-torsors over $B$, then $H^2(B, \mathcal{P})\cong\pi_0(\text{Tor}_{(B, T)})$.  

\begin{definition}\label{charform}
	A 2-form $\Theta$ is called a \emph{characteristic form} of a regular twisted Poisson manifold $B$ if its restriction to any leaf is its almost symplectic form. 
\end{definition}
For the existence of characteristic forms, see the proof of \cite[Proposition 3.1]{BG-N}, which shows that forms satisfying the conditions in Definition \ref{charform} are not unique. As we are going to take the almost symplectic structure into account in solving the classification problem, we shall first consider a subcategory of ASCIRs over $B$. 
\begin{definition}
	Fixing a choice of the characteristic form $\Theta$, we let $\text{ASCIR}_0(B, \Pi, P, \Theta)$ be the category of ASCIRs over $(B, \Pi)$ with period bundle $P$ and twisting form $d\Theta$.  
\end{definition}
\begin{lemma}\label{cotangentiso}
	Suppose $\pi: (M, \omega)\to (B, \Pi)$ admits a global section $s$. Let $\varphi_s: T\to M$ be defined by $\varphi_s(b, \xi)=\varphi(\xi)(s(b))$. Then $\varphi_s$ is $T$-equivariant, and $\varphi_s^*\omega=\omega_T+\pi_T^*s^*\omega$.
\end{lemma}
\begin{proof}
	$T$-equivariance of $\varphi_s$ follows from its definition. Let $\xi\in\Gamma(\mathcal{T})$. We observe that $\varphi_s\circ \xi=\varphi(\xi)\circ s$. Then
	\begin{align*}
		&\xi^*(\varphi_s^*\omega-\omega_T-\pi_T^*s^*\omega)\\
		=&s^*\varphi(\xi)^*\omega-d\xi-s^*\omega\\
		&(d\xi\text{ is well-defined because any representatives of }\xi\text{ in }\Omega_\mathcal{F}^1\\
		&\text{ differ from each other by a closed form represented by a global section of }P)\\
		=&s^*(\omega+\pi^*d\xi)-d\xi-s^*\omega\\
		=&0
	\end{align*}
	To show that $\varphi_s^*\omega-\omega_T-\pi_T^*s^*\omega$ is 0, it remains to show that it is basic. Let $a_1, \cdots, a_n$ be the local action coordinates in an open neighborhood $U$ of $B$. Then $\{X_{\pi_T^*a_i}:=\omega_T^\sharp(\pi_T^*da_i)\}_{1\leq i\leq n}$ spans the tangent spaces of the fibers of $T$ over $U$. So it suffices to show that $\iota_{X_{\pi_T^*a_i}}(\varphi_s^*\omega-\omega_T-\pi_T^*s^*\omega)=0$. On the one hand, we have $\iota_{X_{\pi_T^*a_i}}\pi_T^*s^*\omega=0$ and $\iota_{X_{\pi_T^*a_i}}\omega_T=\pi_T^*da_i$. On the other hand, 
	\[\varphi_s(b, \xi+tda_i)=\varphi(\xi+tda_i)(s(b))\]
	Differentiating with respect to $t$ at $t=0$ yields
	\begin{align*}
		\varphi_{s*}(X_{\pi_T^*a_i})_{(b, \xi)}&=\left.\frac{d}{dt}\right|_{t=0}\varphi(\xi+tda_i)(s(b))\\
		                                                         &=\omega^\sharp(\pi^*da_i)_{\varphi(\xi)(s(b))}
	\end{align*}
	So $\iota_{X_{\pi_T^*a_i}}\varphi_s^*\omega=\iota_{\omega^\sharp(\pi^*da_i)}\omega=\pi^*da_i$, and this completes the proof.
\end{proof}
\begin{lemma}\label{localiso}
	Any $M\in \text{ASCIR}_0(B, \Pi, P, \Theta)$ is $T$-equivariantly, locally almost symplectomorphic to $(T, \omega_T+\pi_T^*\Theta)$.
\end{lemma}
\begin{proof}
	Take any local section $s': U\to M$. Then $\varphi_{s'}^*\omega=\omega_{T|_U}+\pi_{T|_U}^*s'^*\omega$ by Lemma \ref{cotangentiso}. On the other hand, $\pi_{T|_U}^*(\Theta-s'^*\omega)$ is obviously closed. $\pi_{T|_U}$ being a submersion, $\Theta-s'^*\omega$ is closed as well. So by Poincar\'e lemma and restricting $U$ to a smaller neighborhood if necessary, $\Theta-s'^*\omega=dt_{ss'}$ for some 1-form $t_{ss'}$ on $U$. Let $s:=\varphi(t_{ss'})s'$. We then have
	\begin{align*}
		&\varphi_{s'}^*(\varphi_s^*)^{-1}(\omega_{T|_U}+\pi_{T|_U}^*\Theta)\\
		=&(\varphi(-t_{ss'}))^*(\omega_{T|_U}+\pi_{T|_U}^*\Theta)\\
		=&\omega_{T|_U}-\pi_{T|_U}^*dt_{ss'}+\pi_{T|_U}^*\Theta\\
		=&\omega_{T|_U}+\pi_{T|_U}^*s'^*\omega\\
		=&\varphi_{s'}^*\omega
	\end{align*}
	It follows that $\varphi_s^*\omega=\omega_{T|_U}+\pi_{T|_U}^*\Theta$. $T$-equivariance of $\varphi_s$ is given in Lemma \ref{cotangentiso}. 
\end{proof}
For $M\in \text{ASCIR}_0(B, \Pi, P, \Theta)$, we define the sheaf of automorphisms $\mathcal{A}ut^M_{(B, \Pi, P, \Theta)}$ on $B$ by 
	\[\mathcal{A}ut^M_{(B, \Pi, P, \Theta)}(U)=\{f\in\text{Diff}(\pi^{-1}(U))|\pi\circ f=\pi, f^*\omega=\omega\}\]
Lemma \ref{localiso} simply says that $\mathcal{A}ut^M_{(B, \Pi, P, \Theta)}$ is isomorphic to $\mathcal{A}ut^{(T, \omega_T+\pi_T^*\Theta)}_{(B, \Pi, P, \Theta)}$ and thus independent of $M$. We shall drop the superscript $M$ from the notation for the sheaf of automorphisms from now on. As will be clear later, $\mathcal{A}ut_{(B, \Pi, P, \Theta)}$ should be thought of as the `structure sheaf' of $B$. 
\begin{proposition}\label{structuresheaf}
	\begin{enumerate}
		\item We have the short exact sequence 
		\[0\longrightarrow \mathcal{O}_P\longrightarrow\Omega_\mathcal{F}^0\stackrel{\varphi\circ d}{\longrightarrow}\mathcal{A}ut_{(B, \Pi, P, \Theta)}\longrightarrow 0\]
		\item $\varphi: \mathcal{K}_\mathcal{F}\longrightarrow\mathcal{A}ut_{(B, \Pi, P, \Theta)}$ is an isomorphism.
	\end{enumerate}
\end{proposition}
\begin{proof}
	\begin{enumerate}
		\item Obviously the kernel of $\varphi\circ d$ is $\mathcal{O}_P$. It remains to show that $\varphi\circ d$ is surjective and the sequence is exact in the middle. For any point $b\in B$, we can find a neighborhood $U$ such that $(\pi_T^{-1}(U), \omega_{T|_U}+\pi_{T|_U}^*\Theta)\cong (U\times T^k, \sum_{i=1}^n da_i\wedge d\alpha_i+\pi_1^*\Theta)$, where $a_i$ are local action coordinates with $da_i\in\mathcal{P}(U)$, and $\alpha_i$ are angle coordinates. By the above remark, we may consider 
		\[\mathcal{A}ut^{(T, \omega_T+\pi_T^*\Theta)}_{(B, \Pi, P, \Theta)}(U)=\{f\in\text{Diff}(\pi_1^{-1}(U))| \pi_1\circ f=\pi_1, f^*(\sum_{i=1}^nda_i\wedge d\alpha_i+\pi_1^*\Theta)=\sum_{i=1}^n da_i\wedge d\alpha_i+\pi_1^*\Theta\}\]
	In local coordinates, 
	\[f(a_1, \cdots, a_n, b_1, \cdots, b_{2d-2n}, \alpha_1, \cdots, \alpha_n)=(a_1, \cdots, a_n, b_1, \cdots, b_{2d-2n}, f_1, \cdots, f_n)\]
	The condition $f^*(\sum_{i=1}^nda_i\wedge d\alpha_i+\pi_1^*\Theta)=\sum_{i=1}^n da_i\wedge d\alpha_i+\pi_1^*\Theta$ then implies that $d\left(\sum_{i=1}^n (f_i-\alpha_i)da_i\right)=0$, which in turn forces $f_i-\alpha_i$ to be functions independent of the coordinates $b_1, \cdots, b_{2d-2n}, \alpha_1, \cdots, \alpha_n$ (though both $\alpha_i$ and $f_i$ are $\mathbb{R}/\mathbb{Z}$-valued functions, their difference $f_i-\alpha_i$ can be lifted to a well-defined function, as $f$ is a local diffeomorphism). It follows that $\sum_{i=1}^n (f_i-\alpha_i)da_i\in \Omega^1_\mathcal{F}(U)$ and $f=\varphi\left(\sum_{i=1}^n (f_i-\alpha_i)da_i\right)$. By Poincar\'e lemma and restricting $U$ further if necessary, $\sum_{i=1}^n (f_i-\alpha_i)da_i$ is exact. This proves the surjectivity of $\varphi\circ d$. If $f$ is identity, then $\sum_{i=1}^n (f_i-\alpha_i)da_i\in\mathcal{P}(U)$ and any of its primitives is in $\mathcal{O}_P(U)$ by definition. This shows the exactness in the middle of the sequence. 
		\item Part (1) implies that $\varphi\circ d: \Omega_\mathcal{F}^0/\mathcal{O}_P\to \mathcal{A}ut_{(B, \Pi, P, \Theta)}$ is an isomorphism. Besides, $\Omega_\mathcal{F}^0/\mathbb{R}$ and $\mathcal{O}_P/\mathbb{R}$ are isomorphic to $\mathcal{Z}_\mathcal{F}^1$ and $P$ through the exterior differential map. Thus $d$ maps $\Omega_\mathcal{F}^0/\mathcal{O}_P\cong(\Omega_\mathcal{F}^0/\mathbb{R})/(\mathcal{O}_P/\mathbb{R})$ isomorphically onto $\mathcal{Z}_\mathcal{F}^1/\mathcal{P}\cong\mathcal{K}_\mathcal{F}$, and $\varphi: \mathcal{K}_\mathcal{F}\to\mathcal{A}ut_{(B, \Pi, P, \Theta)}$ is indeed an isomorphism. 
	\end{enumerate}
\end{proof}
\begin{remark}
	By (2) of Proposition \ref{structuresheaf} the isomorphism class of $\mathcal{A}ut_{(B, \Pi, P, \Theta)}$ is actually independent of the characteristic form. 
\end{remark}
\begin{theorem}\label{classificationzerotwist}
	\begin{enumerate}
		\item The ASCIRs in $\text{ASCIR}_0(B, \Pi, P, \Theta)$ are classified, up to isomorphism which descends to identity on $B$, by the sheaf cohomology group $H^1(B, \mathcal{A}ut_{(B, \Pi, P, \Theta)})$, which is isomorphic to $H^1(B, \mathcal{K}_\mathcal{F})$. When the leaves of $B$ are points, then $H^1(B, \mathcal{K}_\mathcal{F})\cong H^2(B, \mathcal{O}_P)$. 
		\item (cf. \cite[Proposition 4.1]{DD})We have the long exact sequence 
		\[\cdots \longrightarrow H^1(B, \mathcal{P})\stackrel{d_{P, *}}{\longrightarrow} H^2(B, \mathcal{F})\longrightarrow H^1(B, \mathcal{K}_\mathcal{F})\longrightarrow H^2(B, \mathcal{P})\stackrel{\partial^2}{\longrightarrow} H^3(B, \mathcal{F})\longrightarrow H^2(B, \mathcal{K}_\mathcal{F})\longrightarrow\cdots\]
		which is induced by the short exact sequence
		\[0\longrightarrow \mathcal{P}\longrightarrow \mathcal{Z}_\mathcal{F}^1\longrightarrow\mathcal{K}_\mathcal{F}\longrightarrow 0\]
		The map $\partial^2: H^2(B, \mathcal{P})\to H^3(B, \mathcal{F})$, called the Dazord-Delzant homomorphism, is the composition of the coefficient homomorphism $H^2(B, \mathcal{P})\to H^2(B, \mathcal{Z}_\mathcal{F}^1)$ followed by the isomorphism $H^2(B, \mathcal{Z}_\mathcal{F}^1)\to H^3(B, \mathcal{F})$. Identifying $H^1(B, \mathcal{P})$ with $H^0(B, \mathcal{T})$ and $H^2(B, \mathcal{F})$ with $H^0(B, \mathcal{Z}_\mathcal{F}^2)$, the map $d_{P, *}$ is induced by the exterior differential map $d_P: \mathcal{T}\to \mathcal{Z}_\mathcal{F}^2$. 
\end{enumerate}
\end{theorem}
\begin{proof}
	\begin{enumerate}
		\item That ASCIRs in $\text{ASCIR}_0(B, \Pi, P, \Theta)$ are classified by $H^1(B, \mathcal{A}ut_{(B, \Pi, P, \Theta)})$ follows from Lemma \ref{localiso} and that $\mathcal{A}ut_{(B, \Pi, P, \Theta)}$ is the sheaf of transition functions of those ASCIRs. That $H^1(B, \mathcal{A}ut_{(B, \Pi, P, \Theta)})$ is isomorphic to $H^1(B, \mathcal{K}_\mathcal{F})$ follows from (2) of Proposition \ref{structuresheaf}. Now suppose $B$ is foliated by points. The short exact sequence in (1) of the same Proposition induces the long exact sequence of sheaf cohomology groups
		\[\cdots\longrightarrow H^k(B, \Omega_\mathcal{F}^0)\longrightarrow H^k(B, \mathcal{A}ut_{(B, \Pi, P, \Theta)})\longrightarrow H^{k+1}(B, \mathcal{O}_P)\longrightarrow H^{k+1}(B, \Omega_\mathcal{F}^0)\longrightarrow\cdots\]
		Since $\Omega_\mathcal{F}^0$ is a fine sheaf in this case, $H^k(B, \mathcal{A}ut_{(B, \Pi, P, \Theta)})$ is isomorphic to $H^{k+1}(B, \mathcal{O}_P)$. 
		\item The short exact sequence $0\longrightarrow \mathcal{P}\longrightarrow \mathcal{Z}_\mathcal{F}^1\longrightarrow\mathcal{K}_\mathcal{F}\longrightarrow 0$ induces the long exact sequence
		\[\cdots\longrightarrow H^1(B, \mathcal{Z}_\mathcal{F}^1)\longrightarrow H^1(B, \mathcal{K}_\mathcal{F})\longrightarrow H^2(B, \mathcal{P})\longrightarrow H^2(B, \mathcal{Z}_\mathcal{F}^1)\longrightarrow H^2(B, \mathcal{K}_\mathcal{F})\longrightarrow\cdots\]
		Noting the isomorphism %$H^k(B, \mathcal{K}_\mathcal{F})\cong H^{k+1}(B, \mathcal{O}_P)$ and 
$H^k(B, \mathcal{Z}_\mathcal{F}^1)\cong H^{k+1}(B, \mathcal{F})$, we get the desired long exact sequence and the interpretation of the Dazord-Delzant homomorphism $\partial^2$ and $d_{P, *}$. 
	\end{enumerate}
\end{proof}
\begin{corollary}\label{classificationzerotwist2}
	Endowing $\pi_0(\text{ASCIR}(B, \Pi, P, \Theta))$ with group and topological structures by means of identifying $\pi_0(\text{ASCIR}(B, \Pi, P, \Theta))$ with $H^1(B, \mathcal{K}_\mathcal{F})$ (cf. Theorem \ref{classificationzerotwist} (1)), we have the short exact sequence
	\[0\longrightarrow H^2(B, \mathcal{F})/d_{P, *}H^1(B, \mathcal{P})\longrightarrow \pi_0(\text{ASCIR}_0(B, \Pi, P, \Theta))\rightarrow\text{ker}(\partial^2)\longrightarrow 0\]
	The group $H^2(B, \mathcal{F})/d_{P, *}H^1(B, \mathcal{P})$, which corresponds to the ASCIRs with twisting form $d\Theta$ and vanishing Chern class, is the identity component of $\pi_0(\text{ASCIR}_0(B, \Pi, P, \Theta))$. 
\end{corollary}
\begin{remark}\label{gaugemap}
	Identifying $H^1(B, \mathcal{K}_\mathcal{F})$ with $\pi_0(\text{ASCIR}_0(B, \Pi, P, \Theta))$, we have that the map $H^2(B, \mathcal{F})\to H^1(B, \mathcal{K}_\mathcal{F})$ in the long exact sequence sends the class $[\gamma]$ to the isomorphism class of the ASCIR $(T, \omega_T+\pi_T^*(\gamma+\Theta))$. The map $H^1(B, \mathcal{K}_\mathcal{F})\to H^2(B, \mathcal{P})$ just forgets the symplectic structure and gives the Chern class of the ASCIR as a $T$-torsor. 
\end{remark}
\begin{corollary}\label{zerocherntwist}
	If an ASCIR $(M, \omega)$ in $\text{ASCIR}_0(B, \Pi, P, \Theta)$ has vanishing Chern class, then it is isomorphic to $(T, \omega_T+\pi_T^*(\Theta+\gamma))$ for some $\gamma\in Z_\mathcal{F}^2$.
\end{corollary}
\begin{proof}
	By inspecting the long exact sequence in Theorem \ref{classificationzerotwist}, we note that $[(M, \omega)]$ is in the image of the map $H^2(B, \mathcal{F})\to H^1(B, \mathcal{K}_\mathcal{F})\cong \pi_0(\text{ASCIR}_0(B, \Pi, P, \Theta))$, which sends $[\gamma]$ to $[(T, \omega_T+\pi_T^*(\Theta+\gamma))]$ by the above Remark.
\end{proof}

\section{Picard group and classification of ASCIRs}\label{picard}
In the last Section we obtain the identification of $\pi_0(\text{ASCIR}_0(B, \Pi, P, \Theta))$, the set of isomorphism classes of ASCIRs over $B$ with period bundle $P$ and twisting form $d\Theta$, with sheaf cohomology group $H^2(B, \mathcal{O}_P)$. This begs the question of whether there exists a natural group law among ASCIRs that makes the identification a group isomorphism. The answer is affirmative. In this Section we shall construct a `tensor product' analogous to that of line bundles by means of a symplectic reduction construction (cf. \cite{Xu}), and hence define the notion of Picard group (see Definition \ref{picardgpdef}). We also restate the main result of \cite{SS}, Theorem \ref{SSsuffcond}, which gives the necessary and sufficient condition for a torsor to be an ASCIR with a certain twisting 3-form. Based on this and results in Section 4 we give our main results on the classification of ASCIRs and the description of Picard groups (see Theorem \ref{classification}). 

\begin{definition}
	Let $\text{ASCIR}(B, \Pi, P)$ be the category of ASCIRs over $B$ with $P$ as the period bundle.
\end{definition}
Let $(M_1, \omega_1)$, $(M_2, \omega_2)\in\text{ASCIR}(B, \Pi, P)$ and $d\omega_i=\pi^*\eta_i$, $i=1, 2$. Fix a choice of the characteristic form $\Theta$. The fiber product $M_1\times_B M_2$ is a $T\times_B T$-torsor. By Theorems \ref{IASHSmain} and \ref{IASHSASCIR}, $\omega_i$ can be written in the local normal form $\sum_{j=1}^n da_j\wedge d\alpha_j^{(i)}+\frac{1}{2}\sum_{j, l}A_{jl}^{(i)}da_j\wedge da_l+\sum_{j<l}B_{jl}^{(i)}da_j\wedge db_l+\frac{1}{2}\sum_{j, l}C_{jl}db_j\wedge db_l$, where $A^{(i)}$, $B^{(i)}$ and $C$ are skew-symmetric and $C$ is nondegenerate. The same is true of $\Theta$, which can be written as $\frac{1}{2}\sum_{j, l}D_{jl}da_j\wedge da_l+\sum_{j<l}E_{jl}da_j\wedge db_l+\frac{1}{2}\sum_{j, l}C_{jl}db_j\wedge db_l$. Let $\pi_i: M_1\times_B M_2\to M_i$ and $\pi': M_1\times_B M_2\to B$ be natural projections. The $(n+2d)\times(n+2d)$ matrix associated to the local normal form of $\pi_1^*\omega_1+\pi_2^*\omega_2-\pi'^*\Theta$ is 
	\[\begin{pmatrix}
		A^{(1)}+A^{(2)}-D& I_n& I_n& B^{(1)}+B^{(2)}-E\\
		-I_n& 0&0&0\\
		-I_n&0&0&0\\
		-(B^{(1)}+B^{(2)}-E)^T&0&0&C	
	\end{pmatrix}\]
	%which is obviously nondegenerate. Thus $\pi_1^*\omega_1+\pi_2^*\omega_2-\pi'^*\Theta$ is an almost symplectic form. 
	Let $T^{-}$ be the Lie group bundle over $B$ whose fibers are anti-diagonal of the fibers of $T\times_B T$. The tangent space of the $T^-$-orbits in $M_1\times_B M_2$ is locally spanned by $\partial_{\alpha_j^{(1)}}-\partial_{\alpha_j^{(2)}}$, $j=1, \cdots, n$, and again by means of local normal forms, $\iota_{\partial_{\alpha_j^{(1)}}-\partial_{\alpha_j^{(2)}}}(\pi_1^*\omega_1+\pi_2^*\omega_2)=0$. So $\pi_1^*\omega_1+\pi_2^*\omega_2$ can be pushed down to a 2-form, which we denote by $\omega_1\oplus\omega_2$, of the orbit space $M_1\times_B M_2/T^-$, which is a $T$-torsor (there is a natural residual $T$-action on the orbit space). $\pi'^*\Theta$ is obviously pushed down to $\pi^*\Theta$. The local normal form of $\omega_1\oplus \omega_2-\pi^*\Theta$ then is the $2d\times 2d$ matrix
	\[\begin{pmatrix}
		A^{(1)}+A^{(2)}-D& I_n& B^{(1)}+B^{(2)}-E\\
		-I_n& 0&0\\
		-(B^{(1)}+B^{(2)}-E)^T&0&C	
	\end{pmatrix}\]
	which is nondegenerate. Hence $\omega_1\oplus\omega_2-\pi^*\Theta$ is an almost symplectic form of $M_1\times_B M_2/T^-$.
\begin{definition}\label{tensorproduct}
	We define a binary operation $\otimes$ on $\text{ASCIR}(B, \Pi, P)$ such that $(M_1, \omega_1)\otimes (M_2, \omega_2)$ is $M_1\times_B M_2/T^-$ equipped with the almost symplectic form $\omega_1\oplus\omega_2-\pi^*\Theta$ with twisting form $\eta_1+\eta_2-d\Theta$. 
\end{definition}

%\begin{lemma}
%	Let $f: (M, \omega')\to (M, \omega)$ be an almost symplectomorphism between two ASCIRs in $\text{ASCIR}(B, \Pi, P)$ which descends to identity on $B$. In local action coordinates $a_1, \cdots, a_n$, local angle coordinates $\alpha_1, \cdots, \alpha_n$ and local coordinates for the leaves $b_1, \cdots, b_{2d-2n}$ as in Theorem \ref{IASHS} (2), we have
%	\[f(a_1, \cdots, a_n, b_1, \cdots, b_{2d-2n}, \alpha_1, \cdots, \alpha_n)=(a_1, \cdots, a_n, b_1, \cdots, b_{2d-2n}, A\alpha_1+c_1, \cdots, A\alpha_n+c_n)\]
%	where $A\in GL(n, \mathbb{Z})$ and $c_1, \cdots, c_n$ are smooth functions independent of the angle coordinates.
%\end{lemma}

\begin{proposition}
	The `tensor product' on $\text{ASCIR}(B, \Pi, P)$ as in Definition \ref{tensorproduct} descends to $\pi_0(\text{ASCIR}(B, \Pi, P))$, the isomorphism classes of ASCIRs on $(B, \Pi)$ with a fixed period bundle $P$.  
\end{proposition}
\begin{proof}
	It suffices to show that if $f:(M_1, \omega_1)\to(M'_1, \omega'_1)$ is an almost symplectomorphism of ASCIRs which descends to identity on the base $(B, \Pi)$, then one can define as well an almost symplectomorphism of ASCIRs 
	\[g: (M_1\otimes M_2, \omega_1\oplus\omega_2-\pi^*\Theta)\to (M'_1\otimes M_2, \omega'_1\oplus \omega_2-\pi^*\Theta)\]
	which descends to identity on the base. Explicitly, let 
	\begin{align*}
		\widetilde{g}: M_1\times_B M_2&\to M'_1\times_B M_2\\
		(m_1, m_2)&\mapsto (f(m_1), m_2).
	\end{align*}
	We shall show that $\widetilde{g}$ descends to a map $g$ between $M_1\otimes M_2$ and $M'_1\otimes M_2$ which is well-defined. Denote the action of $\alpha\in T$ on $m\in M$ by $\alpha\cdot_M m$. Noting that $f^*\omega'_1=\omega_1$, we have that the vector fields $\omega'^\sharp_1(\pi^*\alpha)$ and $\omega_1^\sharp(\pi^*\alpha)$ are related by 
	\[f_*(\omega^\sharp_1(\pi^*\alpha))=\omega'^\sharp_1(\pi^*\alpha)\]
	Recalling that the $T$-action on $M_1$ (resp. $M'_1$) is given by the time 1 flow of the vector field $\omega^\sharp_1(\pi^*\alpha)$ (resp. $\omega'^\sharp_1(\pi^*\alpha)$), we get 
	\[f(\alpha\cdot_{M_1} m_1)=\alpha\cdot_{M'_1}f(m_1)\]
	So $\widetilde{g}$ respects the $T^-$-action on $M_1\times_B M_2$ and $M'_1\times_B M_2$, and descends to $g$ which is well-defined as claimed. Lastly, $g$ is obviously an almost symplectomorphism. This completes the proof of the Proposition. 
\end{proof}
\begin{proposition}\label{abgrpstructure}
	$\pi_0(\text{ASCIR}(B, \Pi, P))$ is an abelian group equipped with the `tensor product' as in Definition \ref{tensorproduct} as the binary operation. The identity element is the isomorphism class of $(T, \omega_T+\pi_T^*\Theta)$ and the inverse of $[(M, \omega)]$ is $[(M, -\omega+2\pi^*\Theta)]$. 
\end{proposition}
\begin{proof}
	For any ASCIR $M$, $T\otimes M=T\times_B M/T^-$ is isomorphic to $M$ as $T$-torsors through the $T$-equivariant map $[(t, m)]\mapsto t\cdot m$ (here $T$ acts on $T\times_B M/T^-$ by $s\cdot[(t, m)]=[(st, m)]$, which also pulls the form $\omega$ back to $\omega_T\oplus \omega$, which in turn is the same as $(\omega_T+\pi_T^*\Theta)\oplus \omega-\pi^*\Theta$. So the class $[(T, \omega_T+\pi_T^*\Theta)]$ is indeed the identity.
	The ASCIR $(M, -\omega+2\pi^*\Theta)$ has a $T$-action opposite to that on $(M, \omega)$, because $(-\omega+2\pi^*\Theta)^\sharp(\alpha)=-\omega^\sharp(\alpha)$ for $\alpha\in\Omega_\mathcal{F}^1$. It follows that $(M, \omega)\otimes (M, -\omega+2\pi^*\Theta)$ is isomorphic to $T$ as $T$-torsors. Moreover the form $\omega\oplus(-\omega+2\pi^*\Theta)-\pi^*\Theta$ corresponds to $\omega_T+\pi^*\Theta$ under the said isomorphism. The isomorphism class of $(M, -\omega+2\pi^*\Theta)$ is indeed the inverse of that of $(M, \omega)$. This completes the proof of the Proposition.
\end{proof}
\begin{definition}\label{picardgpdef}
	We call the abelian group $\pi_0(\text{ASCIR}(B, \Pi, P))$ as in Proposition \ref{abgrpstructure} the \emph{Picard group} of the twisted Poisson manifold $(B, \Pi)$ with a fixed characteristic form $\Theta$ and period bundle $P$, and is denoted by $\text{Pic}(B, \Pi, P, \Theta)$. 
\end{definition}
\begin{remark}
	For any two characteristic forms $\Theta$ and $\Theta'$, there is a group isomorphism $\text{Pic}(B, \Pi, P, \Theta)\to\text{Pic}(B, \Pi, P, \Theta')$ which sends $[(M, \omega)]$ to $[(M, \omega+\pi^*(\Theta'-\Theta))]$. So the isomorphism class of the Picard group is independent of the choice of characteristic forms.
\end{remark}
\begin{corollary}\label{homo}
	The map $\text{Pic}(B, \Pi, P, \Theta)\to Z_\mathcal{F}^3(B)$ defined by $[(M, \omega)]\mapsto d\Theta-\eta$ is a group homomorphism. 
\end{corollary}
The next result specifies a necessary topological condition any ASCIR has to satisfy.
\begin{proposition}(cf. \cite[Proposition 5]{SS}) \label{necesscond}The following square is commutative
	\begin{displaymath}
		\xymatrix{\text{Pic}(B, \Pi, P, \Theta)\ar[r]\ar[d]& Z_\mathcal{F}^3(B)\ar[d]\\ H^2(B, \mathcal{P})\ar[r]^{\partial^2}& H^3(B, \mathcal{F})}
	\end{displaymath}
	where the left vertical map is the Chern class map and the right one takes a 3-form to its relative cohomology class.
\end{proposition}
\begin{proof}
	Recall that $c(M)$ is represented by the Cech 1-cocycle $t$ such that $s_\alpha=\varphi(t_{\alpha\beta})s_\beta$ for some local sections $s_\alpha$ and $s_\beta$ on $U_\alpha$ and $U_\beta$ respectively, and $t_{\alpha\beta}\in\Omega_\mathcal{F}^1(U_{\alpha\beta})$. If we identify $H^2(B, \mathcal{P})$ with $H^1(B, \mathcal{T})$ and $H^3(B, \mathcal{F})$ with $H^1(B, \mathcal{Z}_\mathcal{F}^2)$, then $\partial^2$ can be identified with $d_{P, *}: H^1(B, \mathcal{T})\to H^1(B, \mathcal{Z}_\mathcal{F}^2)$. It suffices to show that $p_{\alpha\beta}:=dt_{\alpha\beta}$ is a 1-cocyle representing the class of $d\Theta-\eta$. Note that $\varphi_{s_\beta}\circ t_{\alpha\beta}=\varphi(t_{\alpha\beta})s_\beta=s_\alpha$ and hence $t_{\alpha\beta}^*\varphi_{s_\beta}^*=s_\alpha^*$. We have
	\begin{align*}
		s_\alpha^*\omega&=t_{\alpha\beta}^*\varphi_{s_\beta}^*\omega\\
		                                 &=t_{\alpha\beta}^*(\omega_T+\pi_T^*s_\beta^*\omega)\\
		                                 &=dt_{\alpha\beta}+s_\beta^*\omega
	\end{align*}
	Let $p_\alpha:=\Theta-s_\alpha^*\omega$ and $p_\beta:=\Theta-s_\beta^*\omega$. Note that $dp_\alpha=d\Theta-s_\alpha^*d\omega=d\Theta-s_\alpha^*\pi^*\eta=d\Theta-\eta=dp_\beta$. Hence the 1-cocycle $p_{\alpha\beta}:=p_\beta-p_\alpha=s_\alpha^*\omega-s_\beta^*\omega$ represents the cohomology class of $d\Theta-\eta$. But we also have shown that $p_{\alpha\beta}=dt_{\alpha\beta}$. This completes the proof.
\end{proof}
In fact, the topological condition in Proposition \ref{necesscond} is also sufficient for a $T$-torsor to be an ASCIR. This is basically the content of the main result in \cite{SS}, which is a generalization of a similar result in \cite{DD} where the special case of symplectically complete isotropic realization is treated.
\begin{theorem}\label{SSsuffcond}(cf. \cite[Theorem 3]{SS})
	If a $T$-torsor $M$ satisfies $\partial^2c(M)=[d\Theta-\eta]\in H^3(B, \mathcal{F})$ for some 3-form $\eta$, then there exists an almost symplectic form $\omega$ on $M$ which makes it an ASCIR with a twisting form $\eta$. In particular, any $T$-torsor can be made an ASCIR. 
\end{theorem}
\begin{proof}
We shall show the last claim that any $T$-torsor can be made an ASCIR, which was not shown in \cite{SS}. Suppose $M$ is a $T$-torsor over $(B, \Pi, P, \Theta)$. Let $\eta_0$ be a twisting 3-form of $(B, \Pi)$, and $\alpha=[d\Theta-\eta_0]-\partial^2c(M)\in H^3(B, \mathcal{F})$. Choose a representative 3-form $\widetilde{\alpha}\in Z^3_\mathcal{F}$ of the class $\alpha$. Then $\partial^2c(M)=[d\Theta-(\eta_0+\widetilde{\alpha})]$, and $\eta_0+\widetilde{\alpha}$ is a twisting 3-form of $(B, \Pi)$ as well (see Remark \ref{twistedpoissonrmk}). Hence $M$ can be equipped with a compatible almost symplectic form which makes it an ASCIR with the twisting form $\eta_0+\widetilde{\alpha}$. 
\end{proof}
We are now in a position to present two descriptions of $\text{Pic}(B, \Pi, P, \Theta)$, derived from the commutative square in Proposition \ref{necesscond}, thereby giving some answers to the problem of classifying ASCIRs. In these two descriptions, we fit $\text{Pic}(B, \Pi, P, \Theta)$ into the short exact sequences built from the top and left maps in that commutative square. 
%It is noteworthy that the kernel of these two maps can be interpreted to be the subgroups of the Picard group consisting of the isomorphism classes of ASCIRs of `degree 0', i.e. those ASCIRs with vanishing Chern class, or those whose twisting form is $d\Theta$. 
\begin{theorem}\label{classification}
	\begin{enumerate}
		\item We have the exact sequence
		\[0\longrightarrow H^1(B, \mathcal{K}_\mathcal{F})\longrightarrow\text{Pic}(B, \Pi, P, \Theta)\longrightarrow Z_\mathcal{F}^3(B)\]
		where the last map is the one in Corollary \ref{homo}.
		\item\label{NS} We have the short exact sequence
		\[0\longrightarrow\Omega_\mathcal{F}^2/d\Omega_\mathcal{F}^1\longrightarrow\text{Pic}(B, \Pi, P, \Theta)\longrightarrow H^2(B, \mathcal{P})\longrightarrow 0\]
		where the last map is the Chern class map. 
	\end{enumerate}
\end{theorem}
\begin{proof}
	Part (1) follows from Theorem \ref{classificationzerotwist}. Suppose $(M, \omega)$ is an ASCIR with zero Chern class. Then $M$ is isomorphic to $T$ as a $T$-torsor. By Proposition \ref{necesscond}, the twisting form of $M$ satisfies $[d\Theta-\eta]=\partial^2c(M)=0$. So $\eta=d(\Theta+\gamma)$ for some $\gamma\in\Omega_\mathcal{F}^2$. Thus $(T, \omega)\in\text{ASCIR}_0(B, \Pi, P, \Theta+\gamma)$. By Corollary \ref{zerocherntwist}, $[(T, \omega)]=[(T, \omega_T+\pi_T^*(\Theta+\gamma+\delta))]$ for some $\delta\in Z_\mathcal{F}^2$. Moreover, the map $\text{Pic}(B, \Pi, P, \Theta)\to H^2(B, \mathcal{P})$ is onto by Theorem \ref{SSsuffcond}. It follows that the sequence
	\[\Omega_\mathcal{F}^2\longrightarrow\text{Pic}(B, \Pi, P, \Theta)\longrightarrow H^2(B, \mathcal{P})\longrightarrow 0\]
	where the first map is as stated in Remark \ref{gaugemap}, is exact. If $[(T, \omega_T+\pi_T^*(\Theta+\gamma))]$ is the identity element in the Picard group, then there exists a diffeomorphism $f$ of $T$ which descends to identity on $B$ such that $f^*\omega_T=\omega_T+\pi_T^*\gamma$. Using local coordinates, we can let $f(a_1, \cdots, a_n, b_1, \cdots, b_{2d-2n}, \alpha_1, \cdots, \alpha_n)=(a_1, \cdots, a_n, b_1, \cdots, b_{2d-2n}, \alpha_1+f_1, \cdots, \alpha_n+f_n)$. Plugging this to $f^*\omega_T=\omega_T+\pi_T^*\gamma$, we get 
	\[\sum_{i=1}^n\sum_{j=1}^nda_i\wedge \left(\frac{\partial f_i}{\partial a_j}da_j+\frac{\partial f_i}{\partial b_j}db_j+\frac{\partial f_i}{\partial \alpha_j}d\alpha_j\right)=\pi_T^*\gamma\]
	which shows that $f_i$ are independent of $\alpha_i$, the angle coordinates. In other words, $f=\varphi(s)$ for some $s\in\Omega_\mathcal{F}^1$. Now $\varphi(s)^*\omega_T=\omega_T+\pi_T^*ds$ and therefore $\gamma=ds$. 
\end{proof}
\begin{definition}
	We denote the image of $\Omega_\mathcal{F}^2/d\Omega_\mathcal{F}^1$ in the exact sequence in (2) of Theorem \ref{classification} by $\text{Pic}^0(B, \Pi, P, \Theta)$.
\end{definition}
The exact sequence in (2) of Theorem \ref{classification} is a generalization of the one in Corollary \ref{classificationzerotwist2}. As $H^2(B, \mathcal{P})$ is a discrete group, $\text{Pic}^0(B, \Pi, P, \Theta)$ is the identity component of $\text{Pic}(B, \Pi, P, \Theta)$. $\text{Pic}^0(B, \Pi, P, \Theta)$ is to the subgroup of Picard group consisting of isomorphism classes of degree 0 invertible sheaves in algebraic geometry what $H^2(B, \mathcal{P})$ is to the N\'eron-Severi group (see \cite[Definitions p.109]{Ha} and \cite[Remark 6.10.3]{Ha} and the references therein for definitions of invertible sheaves and N\'eron-Severi groups). This analogy first appeared in \cite{Sj}. 

There is an action of $\Omega_\mathcal{F}^2$ on $\text{ASCIR}(B, \Pi, P)$, namely, $\gamma\cdot(M, \omega)=(M, \omega+\pi^*\gamma)$. We call this action the \emph{coarse gauge transformation}. If two ASCIRs are in the same orbit of coarse gauge transformation, we say they are coarse gauge equivalent. Obviously this equivalence relation descends to the Picard group. Part (2) of Theorem \ref{classification} just says that there is a bijection between the coarse gauge equivalence classes and the isomorphism classes of $T$-torsors, and that any coarse gauge equivalence class is a principal homogeneous space of $\Omega_\mathcal{F}^2/d\Omega_\mathcal{F}^1$. 

\section{An example}\label{twistedlieex}

This section is devoted to an example inspired by an SCIR over a certain open subset in $\mathfrak{g}^*$ with the Lie-Poisson structure studied in \cite{DD} (cf. Remark \ref{twistedlieexrmk}). The twisted Poisson manifold in this example was discussed in \cite{SW}.
\subsection{Twisted Poisson structure of an open subset of a compact Lie group}
Let $G$ be a compact connected Lie group of rank $n$ and $\langle\cdot , \cdot\rangle$ be a bi-invariant inner product on $\mathfrak{g}$. A $G$-invariant closed 3-form on $G$ is given by $\eta=\langle\theta^L, [\theta^L, \theta^L]\rangle$, where $\theta^L$ is the left-invariant Maurer-Cartan connection 1-form. In \cite{SW}, a twisted Dirac structure of $G$ with twisting form $\eta$ (called the \emph{Cartan-Dirac structure}) is given by the following maximal isotropic subbundle $L$ of the Courant algebroid $TG\oplus T^*G\cong TG\oplus TG$ (we identify $TG$ with $T^*G$ by the inner product, so the bilinear form on $TG\oplus TG$ is given by $((X_1, Y_1), (X_2, Y_2))=\langle X_1, Y_2\rangle+\langle Y_1, X_2\rangle$).
\[L=\{(X^R-X^L, \frac{1}{4}(X^R+X^L))\in TG\oplus TG| X\in \mathfrak{g}\}\]
Here $X^L$ (resp. $X^R$) is the left-(resp. right-)invariant vector field generated by $X\in\mathfrak{g}$. Let $G_0=\{g\in G| \text{Ad}_g+1\text{ is invertible}\}$. Noting that $\{(X^R+X^L)_g| X\in\mathfrak{g}\}=T_gG$ if and only if $g\in G_0$, we can define the almost Poisson structure
\begin{align*}
	\Pi^\sharp: TG_0&\to TG_0\\
	v&\mapsto 4L_g\circ(1-\text{Ad}_g)\circ(1+\text{Ad}_g)^{-1}\circ L_{g^{-1}}v
\end{align*}
whose graph is exactly $L$ restricted to $G_0$. In fact, $\Pi$ is an $\eta$-twisted Poisson structure. The vector field $X^R-X^L$ being the fundamental vector field of the conjugation action, the almost symplectic leaves are precisely the conjugacy classes of $G$ in $G_0$. For example, if $G=SU(2)$, $G\setminus G_0$ is the conjugacy class containing $\begin{pmatrix}i &0\\ 0&-i\end{pmatrix}$. If we identify $G$ with $S^3$, then $G\setminus G_0$ can be identified with the equatorial 2-sphere, and $G_0$ consists of two connected components of hemispheres of $S^3$. The almost symplectic leaves are regular conjugacy classes (except $G\setminus G_0$), which are all diffeomorphic to $S^2$, and the two one-element conjugacy classes $\{I_2\}$ and $\{-I_2\}$.
%The following result, which describes the almost symplectic forms of conjugacy classes, should be well-known. Since we are not able to find it in the literature, we shall give a proof of it. 
\begin{proposition}(cf. \cite[Equation (40)]{ABM}, \cite[Proposition 3.1]{AMM} and \cite{GHJW})
	The conjugacy classes of $G$ in $G_0$ together with the following data
	\begin{enumerate}
		\item the $G$-action by conjugation, 
		\item the almost symplectic forms induced by the $\eta$-twisted Poisson structure $\Pi^\sharp$, and
		\item the inclusion map $\iota: \mathcal{C}\to G$ as the $G$-valued moment map
	\end{enumerate}
	are quasi-Hamiltonian manifolds as in \cite{AMM}. The almost symplectic form of a conjugacy class $\mathcal{C}$ in $G_0$ is given by 
	\[(\omega_\mathcal{C})_g(X_g^\sharp, Y_g^\sharp)=\frac{1}{2}(\langle\text{Ad}_g X, Y\rangle-\langle X, \text{Ad}_g Y \rangle)\]
	where $X^\sharp$ is the fundamental vector field induced by the infinitesimal conjugation action by $X\in\mathfrak{g}$.  
\end{proposition}
Let $G_0^{\text{reg}}$ be the set of regular elements in $G_0$. We fix a maximal torus $T$ of $G$ and a Weyl alcove $\Delta$, which parametrizes the set of conjugacy classes in $G$. Let $\mu: G\to \Delta$ be the point-orbit projection map, and $X\in \Delta$ such that $\text{exp}(X)\in G\setminus G_0$. Then $-1$ must be one of the eigenvalues of $\text{Ad}_{\text{exp}(X)}$, which are $e^{\alpha(X)}$ for $\alpha$ in the set of roots $R$. It follows that the image of $G\setminus G_0$ under $\mu$ is the intersection of $\Delta$ with the union of finitely many hyperplanes $\bigcup_{\alpha\in R}\alpha^{-1}(\pi i)$ in $\mathfrak{t}$. Hence $\text{Im}(\mu|_{G_0^{\text{reg}}})=\text{int}(\Delta)\setminus \bigcup_{\alpha\in R}\alpha^{-1}(\pi i)$, which is a disjoint union of finitely many open polytopes. For definiteness we choose the open polytope $Q$ with one vertex being 0, and let $B=\mu^{-1}(Q)$, which is the twisted Poisson manifold we consider in the example. $B\to Q$ is a fiber bundle whose fibers are all regular conjugacy classes, which are diffeomorphic to $G/T$. As $Q$ is contractible, this bundle is trivial. 

\begin{definition}
 Let $\kappa$ be the diffeomorphism $G/T\times Q\to B$ defined by $(gT, X)\mapsto g\text{exp}(X)g^{-1}$. 
\end{definition}

Using $\kappa$ to identify $B$ and $G/T\times Q$, we have that $\nu^*\mathcal{F}$ is isomorphic to the trivial bundle $\pi_2^*T^*Q=G/T\times T^*Q$. Furthermore, 
\begin{lemma}
	$B$ is a tropical affine manifold. Moreover, any basis of $\mathfrak{t}^*$ gives rise to a tropical affine structure of $B$. 
\end{lemma}
\begin{proof}
	By Proposition \ref{tropaffstr}, a tropical affine structure is given by a $\mathbb{Z}^n$-subbundle of $\nu^*\mathcal{F}$ whose sections are closed 1-forms. Given a basis $\{\xi_1, \cdots, \xi_n\}$ of $\mathfrak{t}^*$, the $\mathbb{Z}^n$-subbundle generated by the sections $d(\xi_1\circ\mu), \cdots, d(\xi_n\circ\mu)$ furnishes the desired tropical affine structure. 
\end{proof}

Consider the connection of $B$ (as a $G/T$-fiber bundle over $Q$) which is the push forward of the canonical connection of $G/T\times Q$ through $\kappa$. For convenience we also call this connection of $B$ the canonical connection. 
\begin{definition}
	Let $\Theta_B$ be the unique characteristic 2-form of $B$ such that 
	\[(\Theta_B)_g(V, W)=(\omega_\mathcal{C})_g(V', W')\]
	where $V'$ and $W'$ are projections onto $T_g\mathcal{C}$ of $V$ and $W$ along the tangent space of the horizontal section through $g$.
\end{definition}

\subsection{An ASCIR over $B$}
To make matters simpler we assume further from now on the following. We let 
\begin{enumerate}
	\item $G$ be simply-connected, 
	\item $\Lambda$ be the integral lattice of $\mathfrak{t}$ and $\Lambda^*=\text{Hom}(\Lambda, \mathbb{Z})$ be the weight lattice,  
	\item $P$ be the tropical affine structure which is induced by a basis $\{\xi_1, \cdots, \xi_n\}$ of $\Lambda^*$ and $\{X_1, \cdots, X_n\}\subset \Lambda$ be the dual basis of $\{\xi_1, \cdots, \xi_n\}$, and
	\item $\langle\cdot, \cdot\rangle$ be the bi-invariant inner product of $\mathfrak{g}$ such that $\{X_1, \cdots, X_n\}$ is orthonormal.
\end{enumerate}
\begin{definition}
	Let $G\times Q\times_T(U(1))_\xi^Y$ be the circle bundle over $G/T\times Q$ with $T$ acting on $U(1)$ with weight $\xi$ and $G/T\times T^*Q$ acting on $G\times Q\times_T(U(1))_\xi^Y$ by 
	\[(gT, X, \zeta)\cdot [(g, X, z)]=[(g, X, \text{exp}(2\pi i\zeta(Y))z)]\]
	Let $M_\xi^Y$ be $(\kappa^{-1})^*(G\times Q\times_T (U(1))_\xi^Y)$, with the $\nu^*\mathcal{F}$-action induced by the $G/T\times T^*Q$-action on $G\times Q\times_T (U(1))_\xi^Y$, and $M_\xi$ the same circle bundle forgetting the $\nu^*\mathcal{F}$-action. 
\end{definition}

By \cite[Lemma 3.1, (3.2)]{Dua} and the fact that $H^i(G)=0$ for any simply-connected compact Lie group $G$ and $i=1, 2$, the transgression for the fibration $T\hookrightarrow G\to G/T$ maps $H^1(T, \mathbb{Z})$ isomorphically onto $H^2(G/T, \mathbb{Z})$. If $H^1(T, \mathbb{Z})$ is identified with $\Lambda^*$, then the transgression maps $\xi$ to $c_1(G\times_T\mathbb{C}_\xi)$. It follows that $\Lambda^*$ can be identified with $H^2(B, \mathbb{Z})$ through the map $\xi\mapsto c_1(M_\xi)$. The `N\'eron-Severi group' $H^2(B, \mathcal{P})$ can then be identified with 
\[\text{Hom}_\mathbb{Z}(\Lambda, H^2(B, \mathbb{Z}))\cong \text{Hom}_\mathbb{Z}(\Lambda, \Lambda^*)\]
Now it is clear that a $\nu^*\mathcal{F}/P$-torsor over $B$ has the Chern class represented by $f\in\text{Hom}_\mathbb{Z}(\Lambda, \Lambda^*)$ if and only if it is isomorphic to 
\[M_{f(X_1)}^{X_1}\times_B M^{X_2}_{f(X_2)}\times_B\cdots\times_B M^{X_n}_{f(X_n)}\]
\begin{proposition}\label{deltatrivtrop}
	\[\partial^2 c(M_{f(X_1)}^{X_1}\times_B M^{X_2}_{f(X_2)}\times_B\cdots\times_B M^{X_n}_{f(X_n)})=\sum_{i=1}^n d(c_1(M_{f(X_i)})\cdot(\xi_i\circ\mu))\]
\end{proposition}
\begin{proof}
	$P$ being a trivial $\mathbb{Z}^n$-bundle, we have a canonical isomorphism $H^2(B, \mathcal{P})\cong H^2(B, \mathbb{Z})\otimes H^0(B, \mathcal{P})$, where $H^0(B, \mathcal{P})$ is spanned by the basis $\{d(\xi_i\circ\mu)\}_{i=1}^n$. The Dazord-Delzant homomorphism then amounts to the composition
	\begin{align*}
		H^2(B, \mathbb{Z})\otimes H^0(B, \mathcal{P})&\longrightarrow H^3(B, \mathcal{F})\\
		\alpha\otimes d(\xi_i\circ\mu)&\mapsto\alpha\cdot d(\xi_i\circ\mu)=d(\alpha\cdot(\xi_i\circ\mu))
	\end{align*}
	Using this interpretation the Proposition follows easily.
\end{proof}
\begin{definition}
	Let $M$ be the $\nu^*\mathcal{F}/P$-torsor $G\times Q$ over $B$ where the projection map is $(g, X)\mapsto g\text{exp}(X)g^{-1}$ and the torsor structure is induced by the $G/T\times T^*Q$-action given by 
	\[(gT, X, \zeta)\cdot(g, X)=(g\text{exp}(-\langle\zeta, \cdot\rangle), X)\]
\end{definition}
It is not hard to see that $M$ is isormphic to the torsor $M^{X_1}_{\xi_1}\times_B M^{X_2}_{\xi_2}\times_B\cdots\times_B M^{X_n}_{\xi_n}$ and hence its Chern class is represented by the map $(X\mapsto \langle X, \cdot\rangle)\in \text{Hom}_\mathbb{Z}(\Lambda, \Lambda^*)$.
\begin{proposition}
	There exists an almost symplectic form on $M$ such that it is an ASCIR over $B$ with the twisting form $\eta$. 
\end{proposition}
\begin{proof}
	It suffices to check that $\partial^2c(M)=[d\Theta_B-\eta]$ by Theorem \ref{SSsuffcond}. Let $D$ be the open subset in $\mathfrak{g}^*$ which is diffeomorphic to $B$ through the exponential map. By \cite{AMM}, there exists a 2-form $\varpi\in\Omega^2(\mathfrak{g}^*)$ satisfying $\text{exp}^*\eta=d\varpi$ and $\omega_\mathcal{O}=\text{exp}^*\omega_\mathcal{C}-i^*_\mathcal{O}\varpi$, where $\mathcal{O}$ is a coadjoint orbit, $\omega_\mathcal{O}$ its Kostant-Kirillov-Soriau 2-form and $\mathcal{C}=\text{exp}(\mathcal{O})$. We have that 
	\[\text{exp}^*(d\Theta_B-\eta)=d(\text{exp}^*\Theta_B-\varpi)\]
On the other hand, by Proposition \ref{deltatrivtrop}
	\[\text{exp}^*\partial^2c(M)=d\left(\sum_{i=1}^n \text{exp}^*c_1(M_{\xi_i})\cdot(\xi_i\circ\mu\circ\text{exp})\right)\]
	Now it boils down to showing that $[\text{exp}^*\Theta_B-\varpi]=\sum_{i=1}^n \text{exp}^*c_1(M_{\xi_i})\cdot(\xi_i\circ\mu\circ\text{exp})\in H^2(D, \mathbb{R})\cong H^2(G/T\times Q, \mathbb{R})$. Restricting to the coadjoint orbit $\mathcal{O}_\zeta$ containing $\zeta\in Q$, we have
	\begin{align*}
		i^*_{\mathcal{O}_\zeta}(\text{exp}^*\Theta_B-\varpi)&=\text{exp}^*i^*_\mathcal{C}\Theta_B-i^*_{\mathcal{O}_\zeta}\varpi\\
		                                                                               &=\text{exp}^*\omega_\mathcal{C}-i^*_{\mathcal{O}_\zeta}\varpi\\
		                                                                               &=\omega_{\mathcal{O}_\zeta}\\
		i^*_{\mathcal{O}_\zeta}\sum_{i=1}^n\text{exp}^* c_1(M_{\xi_i})\cdot(\xi_i\circ\mu\circ\text{exp})&=\sum_{i=1}^n[\omega_{\mathcal{O}_{\xi_i}}]\cdot (\xi_i(\zeta))\\
		                                                                                                                                                                        &=[\omega_{\mathcal{O}_\zeta}]
	\end{align*}
	So $i^*_{\mathcal{O}_\zeta}([\text{exp}^*\Theta_B-\varpi]-\sum_{i=1}^n\text{exp}^* c_1(M_{\xi_i})\cdot(\xi_i\circ\mu\circ\text{exp}))=0$. By applying K\"unneth formula to $H^2(G/T\times Q, \mathbb{R})$ and the acyclicity of $H^*(Q, \mathbb{R})$, we have that $[\text{exp}^*\Theta_B-\varpi]=\sum_{i=1}^n \text{exp}^*c_1(M_{\xi_i})\cdot(\xi_i\circ\mu\circ\text{exp})$, and the desired conclusion follows.
\end{proof}
\begin{remark}\label{twistedlieexrmk}
	The example of ASCIR discussed in this Section should be viewed as the `exponentiated version' of the example of SCIR over the union of regular coadjoint orbits in $\mathfrak{g}^*$ equipped with the Lie-Poisson structure in \cite{DD}. The difference between the Lie-Poisson structure of $D$ and the twisted Poisson structure of $B$ is accounted for by the introduction of the 2-form $\varpi$, whose exterior differential is the twisted 3-form $\eta$. This example is also inspired by the AMM action groupoid $G\rtimes G$, where $G$ acts on itself by conjugation and which integrates the Cartan-Dirac structure (cf. \cite[Section 7]{BCWZ} and \cite{BXZ}). 
\end{remark}
\section{Some general remarks}\label{remarks}
	All results in this paper fits into the framework of integration of twisted Poisson or, even more generally, twisted Dirac structures (cf. \cite{BCWZ}, \cite{CF1}, \cite{CF2} and \cite{CX}). Viewing $B$ more generally as a twisted Dirac manifold, we see that the ASCIRs over $B$ are nothing but special examples of \emph{presymplectic realizations} (cf. \cite[Definition 7.1]{BCWZ}). If we regard $B$ with $\Pi=0$ and an exact twisting 3-form $d\xi$ as a Lie algebroid, then $(T^*B, \omega_{\text{can}}+\pi_{T^*B}^*\xi)$ is an \emph{almost symplectic groupoid} with source and target maps being identical and fibers of the source map being simply-connected, which integrates $B$, in the sense of Definition 2.1 of \cite{CX}. If in addition $B$ possesses a tropical affine structure $P$, then $(T, \omega_T+\pi_T^*\xi)$ is a proper almost symplectic groupoid. 
	
	The notions of tensor product and Picard groups are inspired by the paper \cite{BW} where they are explored in the more general framework of \emph{symplectic dual pairs}. Using their terminology $\text{Pic}(B, \Pi, P, \xi)$ is the \emph{static} Picard group of the groupoid $(T, \omega_T+\pi^*_T\xi)$ if $\Pi=0$.

\footnotesize\textsc{Department of Mathematics, Cornell University, Ithaca, NY 14853, USA \\
\\
National Center for Theoretical Sciences, Mathematics Division, National Taiwan University, Taipei, Taiwan\\E-mail: }\texttt{ckfok@ntu.edu.tw}\\
\textsc{URL: }\texttt{http://www.math.cornell.edu/$\sim$ckfok}
\end{document}